\documentclass[12pt,reqno]{article}

\usepackage[usenames]{color}
\usepackage{amssymb}
\usepackage{graphicx}
\usepackage{amscd}

\usepackage[colorlinks=true,
linkcolor=webgreen,
filecolor=webbrown,
citecolor=webgreen]{hyperref}

\definecolor{webgreen}{rgb}{0,.5,0}
\definecolor{webbrown}{rgb}{.6,0,0}

\usepackage{color}
\usepackage{fullpage}
\usepackage{float}
\usepackage[shortlabels]{enumitem}

\usepackage{graphics,amsmath,amssymb}
\usepackage{amsthm}
\usepackage{amsfonts}
\usepackage{latexsym}
\usepackage{epsf}

\newcommand{\RR}{\mathbb{R}}
\newcommand{\ZZ}{\mathbb{Z}}

\newcommand{\NN}{\mathbb{N}}
\newcommand{\bounded}{{\cal S}}

\newcommand{\ag}{{\rm ag}}

\begin{document}

\theoremstyle{plain}
\newtheorem{theorem}{Theorem}
\newtheorem{corollary}[theorem]{Corollary}
\newtheorem{lemma}[theorem]{Lemma}
\newtheorem{proposition}[theorem]{Proposition}

\theoremstyle{definition}
\newtheorem{definition}[theorem]{Definition}
\newtheorem{example}[theorem]{Example}
\newtheorem{conjecture}[theorem]{Conjecture}

\theoremstyle{remark}
\newtheorem{remark}[theorem]{Remark}

\title{Badly approximable numbers, Kronecker's theorem, and diversity of
Sturmian characteristic sequences}

\author{Dmitry Badziahin\\
School of Mathematics and Statistics\\
University of Sydney\\
NSW 2006 \\
Australia \\
{\tt dzmitry.badziahin@sydney.edu.au} \\
\and
Jeffrey Shallit\thanks{Supported by NSERC Grant 2018-04118.} \\
School of Computer Science\\
University of Waterloo\\
Waterloo, ON  N2L 3G1 \\
Canada \\
{\tt shallit@uwaterloo.ca} }

\maketitle

\begin{abstract}
We give an optimal version of the classical
``three-gap theorem'' on the fractional
parts of $n \theta$, in the case where $\theta$ is an irrational number that is
badly approximable.  As a consequence, we deduce a version of
Kronecker's inhomogeneous approximation theorem in one dimension
for badly approximable numbers.  We apply these results to obtain an
improved measure of sequence diversity for characteristic Sturmian
sequences, where the slope is badly approximable.
\end{abstract}

\section{Introduction}

About twenty-five years ago, the second author \cite{Shallit:1996} published
an article in which several results about numbers with bounded partial
quotients were proved.  In this paper we improve those results ---
in some cases, optimally.

Let us recall what was proved previously.  If $\theta$ is a real number with
simple continued fraction $\theta = [a_0, a_1, a_2, \ldots ]$, then we say
that $\theta$ has {\it bounded partial quotients} or is {\it badly
approximable} if there exists a positive integer $B$ such that $|a_i| \leq B$
for all $i \geq 1$.  The set of real numbers with partial quotients bounded
by $B$ is denoted by $\bounded_B$.  For a survey about such numbers and their
properties, see \cite{Shallit:1992}.

{\it Kronecker's theorem} is a celebrated theorem about inhomogeneous
Diophantine approximation.  In the one-dimensional version, it
states that if $\theta$ is an irrational real number, $\beta$ is
a real number, and $N$ and $\epsilon$ are positive real
numbers, there exist integers $n, p$ with $n > N$ such that
$$ |n \theta - \beta - p | < \epsilon.$$
See, for example, \cite[Chap.~23]{Hardy&Wright:1985}.
Note that Kronecker's theorem provides no estimate of the sizes of the
numbers $n, p$, and indeed, no such estimate is possible, in general, since
(for example) the ratio $\beta/\theta$ could be arbitrarily large.

However, if $\theta$ is badly approximable and $\beta$ is bounded, then it is possible
to bound $n$ and $p$.  We recall Theorem 17 from \cite{Shallit:1996}:

\begin{theorem}\label{th1}
Let $\theta$ be an irrational real number, $0 < \theta < 1$, with
partial quotients bounded by $B$.  Let $0 \leq \beta < 1$ be a real
number.  Then for all $N \geq 1$ there exist integers $p, q$ with
$0 \leq p, |q| \leq (B+2) N^2$ such that $| p\theta - \beta - q |\leq {1 \over N}$.
\end{theorem}

In this paper we improve the upper bound $(B+2)N^2$ to $\frac{f(B)}{2} N$,
where $f$ is a certain function that is bounded above by $(1+\sqrt{4/5})B
< 2B$.

Our main tool is an optimal and apparently new
estimate, for badly approximable numbers,
on the size of
largest interval in the celebrated ``three-gap theorem'' (aka the
Steinhaus conjecture)
\cite{Slater:1950,Florek:1951,Sos:1957,Suranyi:1958,Swierczkowski:1958,Halton:1965,Alessandri&Berthe:1998}.  This is done in Section~\ref{threegap}.

Finally, we apply these results to prove a new measure of sequence diversity
for the so-called characteristic Sturmian sequences.  Roughly speaking, this
measure shows that linearly-indexed subsequences of Sturmian sequences cannot
agree for ``too long''.

\section{The three-gap theorem}
\label{threegap}

Let us begin by recalling the three-gap theorem.  For a real irrational
number $\theta$, let $\{ \theta \}$ denote its fractional part, which can
also be written $\theta \bmod 1$.  Let $|| \theta ||$ denote the distance
from $\theta$ to the nearest integer, which is $\min\{\theta \bmod 1,
(-\theta) \bmod 1\}$.

Let $N$ be a positive integer, and sort the $N+2$ numbers
$$ 0, \{ \theta \}, \{ 2 \theta \}, \ldots, \{ N \theta \}, 1$$
in ascending order, viz.,
$$ s_0 = 0 < s_1 < s_2 < \cdots < s_N < 1 = s_{N+1} .$$
The {\it gaps\/} are the numbers $s_{i+1} - s_i$ for $0 \leq i \leq N$,
and the {\it gap set\/}
$G(\theta, N)$ is the set $\{ s_{i+1} - s_i \ : \ 0 \leq i \leq N \}$.
One version of the three-gap theorem is as follows:

\begin{theorem}
For integers $N \geq 1$ the set
$G(\theta, N)$ is always of cardinality either two or three, and if it is
of cardinality three, the larger of the three numbers is the sum of the
smaller two.
\end{theorem}

Suppose the continued fraction expansion of $\theta$ is $[a_0; a_1, \ldots
]$, and the convergents to $\theta$ are $p_i/q_i$ for $i \geq 0$. In
addition, define the following notation for irrational $\theta $:
$$
\theta_k:= [0;a_k,a_{k+1}, a_{k+2},\ldots];\quad
\phi_k:=[0;a_k,a_{k-1},\ldots, a_1].
$$
Let us recall some basic formulae about continued fractions, most of which
can be found in \cite[Chap.~10]{Hardy&Wright:1985}:
\begin{equation}\label{eq_cf4}
\phi_k = \frac{q_{k-1}}{q_k};
\end{equation}
\begin{equation}\label{eq_cf1}
q_k||q_{k-1}\theta|| + q_{k-1}||q_k\theta|| = 1;
\end{equation}
\begin{equation}\label{eq_cf2}
q_k||q_k\theta|| = \frac{1}{a_{k+1}+\theta_{k+2}+\phi_k};
\end{equation}
\begin{equation}\label{eq_cf3}
q_{k}||q_{k-1}\theta|| = \frac{1}{1+\theta_{k+1}\phi_k}.
\end{equation}

It turns out that the gaps are quantifiable in terms of the continued
fraction for $\theta$. More precisely, Van Ravenstein
\cite{van.Ravenstein:1988} proved
\begin{theorem}
$$
G(\theta,N) \subseteq
\begin{cases}
\{||q_k\theta||, \ 
||q_{k-1}\theta||, \ ||q_k\theta||+||q_{k-1}\theta||\}, & \\
\quad\quad\quad\quad	\text{if $q_k\le N< q_k+q_{k-1}$};\\
\{||q_{k}\theta||,\ ||q_{k-1}\theta||-(l-1)||q_k\theta||,\ 
||q_{k-1}\theta|| - l||q_k\theta||\}, & \\
\quad\quad\quad\quad \text{ if $lq_k+q_{k-1}\le N< (l+1)q_k+q_{k-1}$ for $0<l<a_{k+1}$}.
\end{cases}
$$
\end{theorem}

Let $H(\theta, N) = \max G(\theta, N)$, the largest gap in the gap set.
For each integer $B\ge 1$ we set
$$
f(B) := \sup_{\theta\in \bounded_B }\; \sup_{N\in\NN} \ N \cdot H(\theta,N).
$$
As we will see soon, the values of $f(B)$ are always finite. Moreover it is
easy to see that $f(B)$ is the smallest value that satisfies the following
property: for all $\theta\in \bounded_B$, for all $N\in\NN$ and for all $g\in
G(\theta, N)$ we have $g\le f(B) / N$. The core result of this paper is

\begin{theorem}\label{th4}
Let $B \geq 1$ be an integer.
Then
\begin{equation}\label{eq_fd}
f(B) = \begin{cases}
\displaystyle 1+ \frac{(a+1)^2}{2\sqrt{a^2+2a}}, &
\text{if } B=2a,\; a\in\NN;\\[2ex]
\displaystyle 1+\frac{a^2+3a+2}{\sqrt{4a^2+12a+5}}, & \text{if } B= 2a+1,\; a\in
\ZZ_{\ge 0} .
\end{cases}
\end{equation}
\label{badthm}
\end{theorem}

In what follows we always assume that $\theta\in \bounded_B$.
First, we prove the following easy lemma:

\begin{lemma}\label{lem1}
$$
\min \bounded_B = [0;\overline{B,1}] = \frac{\sqrt{B^2+4B} - B}{2B};
$$$$
\max \bounded_B  = [\overline{B;1}] = \frac{\sqrt{B^2+4B}+B}{2}.
$$
\end{lemma}

\begin{proof}
The result immediately implies from the following fact: as a real function of
its partial quotients, the continued fraction
$$
\theta = [a_0;a_1,a_2,\ldots]
$$
is monotonically increasing in the even-numbered partial quotients $a_0,
a_2,\ldots$, and is monotonically decreasing in the odd-numbered partial
quotients $a_1, a_3, \ldots$.
\end{proof}

\begin{proof}[Proof of Theorem~\ref{badthm}.]

\ \vphantom{a}\\

\noindent{\it Case 1.} Consider $q_k\le N< q_{k}+q_{k-1}$. Then we have
$$
\sup_{q_k\le N <q_k+q_{k-1}} \ N  \cdot H(\theta, N) \le
(q_k+q_{k-1}-1)(||q_k\theta||+||q_{k-1}\theta||).
$$
Moreover, this upper bound is sharp. We first estimate the value of
\begin{equation}
(q_k+q_{k-1})(||q_k\theta||+||q_{k-1}\theta||).
\label{sum1}
\end{equation}
Expanding and then simplifying \eqref{sum1}, we get,
using \eqref{eq_cf1} and \eqref{eq_cf2}, that
\begin{equation}\label{eq2}
q_k||q_k\theta||\! + q_{k-1}||q_{k-1}\theta||\! +
q_k||q_{k-1}\theta||\!+q_{k-1}||q_k\theta||\!
=\!
1+
\frac{1}{a_{k+1}+\theta_{k+2}+\phi_k}+\frac{1}{a_k+\theta_{k+1}+\phi_{k-1}}.
\end{equation}
Introduce two variables $x = a_{k+1} + \theta_{k+2}$ and $y = \phi_k$. Then
it is easy to check that $\theta_{k+1} = 1/x$ and $a_k + \phi_{k-1} = 1/y$.
Finally the expression \eqref{sum1} simplifies to
\begin{equation}\label{eq1}
1 + \frac{1}{x+y} + \frac{1}{x^{-1}+y^{-1}} = 1+ \frac{xy+1}{x+y}.
\end{equation}
By looking at the partial derivatives of the right-hand side of \eqref{eq1},
we see that it grows monotonically in $x$ when $y>1$ and decreases
monotonically when $0<y<1$. By symmetry the same is true for $y$. Now, since
$x>1$ and $y<1$,~\eqref{eq1} is maximized when $y$ is maximized and $x$ is
minimized. Since $\theta\in \bounded_B $, Lemma~\ref{lem1} says that the
minimum possible $x$ is $\frac{\sqrt{B^2+4B}+B}{2B}$. Among all values $y'\in
\bounded_B$ the maximum possible is $y' = \frac{\sqrt{B^2+4B}-B}{2}$. Then
$xy'=1$ and
$$
x+y' = \frac{(B+1)\sqrt{B^2+4B} +B-B^2}{2B}.
$$
Finally,
$$
1+\frac{xy'+1}{x+y'} = 1+\frac{4B}{(B+1)\sqrt{B^2+4B}-B^2+B} =
1+\frac{(B+1)\sqrt{B^2+4B} +B^2-B}{2B^2 +2B +1}.
$$

However, in our case, $y$ is rational with partial quotients bounded by
$B$. This value can be slightly bigger than $y'$.  On the other
hand, for all such $y$ there exists $\xi\in \bounded_B$ such that $y$ is a
convergent of $\xi$. Since $y = \phi_k = \frac{q_{k-1}}{q_k}$, this
implies that for the maximum possible $y$,
$$
y - y' < \frac{1}{q_k^2}.
$$
Then we have
$$
\frac{xy+1}{x+y} - \frac{xy'+1}{x+y'} < \frac{\partial}{\partial
y}\left(\frac{xy''+1}{x+ y''}\right)(y-y') =
\frac{x^2-1}{(q_k(x+y''))^2},
$$
where $y''$ is some value between $y$ and $y'$. Next, we have that
$x<1+1/B$ and $x+y' > 2$. Therefore
$$
\frac{xy+1}{x+y} - \frac{xy'+1}{x+y'} < \frac{2B+1}{4B^2 q_k^2} <
||q_k\theta|| + ||q_{k-1}\theta|| .
$$
Finally, we have that
\begin{equation}
\sup_{q_k\le N <q_k+q_{k-1}} \ N \cdot H(\theta, N) \le 1+
\frac{(B+1)\sqrt{B^2+4B} +B^2-B}{2B^2 +2B +1}.
\end{equation}

Note that this estimate is sharp. That is, for all $\epsilon>0$, one can find
$\theta\in \bounded_B$ and $k\in\NN$ such that
$$
\sup_{q_k\le N <q_k+q_{k-1}} \ N \cdot H(\theta, N) > 1+
\frac{(B+1)\sqrt{B^2+4B} +B^2-B}{2B^2 +2B +1}-\epsilon.
$$
Therefore we have
\begin{equation}\label{ineq_case1}
\sup_{\theta\in \bounded_B}\;\sup_{k\in \NN}\; \max_{q_k\le N<q_k+q_{k-1}}
N\cdot H(\theta,N) = 1+ \frac{(B+1)\sqrt{B^2+4B} +B^2-B}{2B^2 +2B
+1}.
\end{equation}

\bigskip

\noindent{\it Case 2.} Assume that $lq_k+q_{k-1}\le N< (l+1)q_k+q_{k-1}$ for
some integer $l$ between 1 and $a_{k+1}-1$. Then we have
$$
\sup_{lq_k+q_{k-1}\le N <(l+1)q_k+q_{k-1}}  \ N \cdot H(\theta,
N) \le ((l+1)q_k+q_{k-1}-1)(||q_{k-1}\theta||-(l-1)||q_k\theta||).
$$
We proceed as in Case 1. Define $x = a_{k+1} +
\theta_{k+2}$ and $y = \phi_k$. We computed the values
$q_k||q_k\theta||$ and $q_{k-1}||q_{k-1}\theta||$ in the previous
case. Also from~\eqref{eq_cf3} we compute
$$
q_k||q_{k-1}\theta|| = \frac{1}{1+\theta_{k+1}\phi_k} =
\frac{x}{x+y}.
$$
Finally from~\eqref{eq_cf1}, we have $q_{k-1}||q_k\theta|| = 1 -
q_k||q_{k-1}\theta|| = \frac{y}{x+y}$. Now we expand
\begin{align*}
((l+1)q_k+q_{k-1})(||q_{k-1}\theta||-(l-1)||q_k\theta||) &=
\frac{1-l^2}{x+y} + \frac{(l+1)x}{x+y} - \frac{(l-1)y}{x+y} +
\frac{xy}{x+y}.\\
& = \frac{xy -(l-1)y + (l+1)x -l^2+1}{x+y}.
\end{align*}
Let the right-hand side be denoted by $F(x,y)$. We look for its maximum.
$$
\frac{\partial}{\partial x} F(x,y) = \frac{y^2 + 2ly
+l^2-1}{(x+y)^2} = \frac{(y+l)^2-1}{(x+y)^2}.
$$
Since $0<y<1$, $F(x,y)$ is monotonically increasing in $x$.
Therefore, by Lemma~\ref{lem1}, the maximum is achieved at $x =
\frac{\sqrt{B^2+4B}+B}{2}$. Next,
$$
\frac{\partial}{\partial y} F(x,y) = \frac{x^2 - 2lx
+l^2-1}{(x+y)^2} = \frac{(x-l)^2-1}{(x+y)^2}.
$$
Notice that $x>B\ge l+1$. Therefore $F(x,y)$ monotonically increases in $y$.
Among all values $y'\in \bounded_B$ the maximum possible is $y' =
\frac{\sqrt{B^2+4B}-B}{2}$. By using the same argument as in case 1, the
maximum rational value $y$ satisfies $y-y' < q_k^{-2}$.

We have found that for each $l$ between $1$ and $B-1$, the maximum value of
$F(x,y)$ is achieved at the same values of $x$ and $y'$ from $\bounded_B$.
Now let us find $l$ for which the value of $F(x,y')= F(x,y',l)$ is maximal.
If we look at $F(x,y')$ as a function of $l$ and look at its partial
derivative we get
$$
\frac{\partial}{\partial l} F(x,y',l) = \frac{-2l+x-y'}{x+y'}.
$$
Notice that $x-y' = B$; therefore $F(x,y,l)$ is maximal for $l = B/2$
when $B$ is even and $l = \frac{B\pm 1}{2}$ if $B$ is odd.
One can easily check that, since $x-y'= B$ the values
$F(x,y',(B+1)/2)$ and $F(x,y',(B-1)/2)$ coincide. So we can pick
either of these two values of $l$; let us choose $l = (B-1)/2$.

\bigskip

\noindent{\it Case 2.1.} Let $B=2a$. Then $l=a$ and
$$
F(x,y',l) = \frac{xy'+x+y'+1+l(x-y') - l^2}{x+y'} = 1+
\frac{(a+1)^2}{2\sqrt{a^2+2a}}.
$$

Now we compute
$$
F(x,y,l)-F(x,y',l) = \frac{(x-l)^2-1}{(x+y'')^2}(y-y')
<\frac{(a+1)^2}{a^2 q_k^2}.
$$
The last is definitely smaller than
$$
||q_{k-1}\theta|| - (l-1)||q_k\theta|| = \frac{x -
(l-1)}{q_k(x+y)}>\frac{a}{2(a+1)q_k}.
$$
Finally we compute
\begin{equation}\label{ineq_case21}
\sup_{\theta\in \bounded_B }\;\sup_{k\in \NN}\; \max_{q_k+q_{k-1}\le
N<q_{k+1}} \ N \cdot H(\theta,N) = 1+ \frac{(a+1)^2}{2\sqrt{a^2+2a}}.
\end{equation}

\noindent{\it Case 2.2.} Let $B=2a+1$. Then $l=a$ and
$$
F(x,y',l) = 1+ \frac{xy+1 + a(2a+1) - a^2}{x+y} =
1+\frac{a^2+3a+2}{\sqrt{4a^2+12a+5}}.
$$
By computations similar to those in Case 2.1, we derive that $F(x,y,l)
- F(x,y',l) < ||q_{k-1}\theta|| - (l-1)||q_k\theta||$ and therefore
\begin{equation}\label{ineq_case22}
\sup_{\theta\in \bounded_B }\;\sup_{k\in \NN}\; \max_{q_k+q_{k-1}\le
N<q_{k+1}} N\cdot H(\theta,N) =
1+\frac{a^2+3a+2}{\sqrt{4a^2+12a+5}}.
\end{equation}

By careful computations one can notice that the right-hand side
of~\eqref{ineq_case1} is always smaller (except the case $B=1$ when
Case 2 is impossible) than the corresponding right-hand sides
in~\eqref{ineq_case21} and~\eqref{ineq_case22}.  From this our result
now follows.
\end{proof}

\begin{corollary}
We have ${B \over 4} \leq f(B) \leq (1+\sqrt{4/5}) B$.
\end{corollary}

\begin{proof}
This is an implication of the following estimates. For $B = 2a$,
from~\eqref{eq_fd} we have
$$
\frac{B}{4}\le 1+\frac{a+1}{2}\le 1+\frac{(a+1)^2}{2\sqrt{(a+1)^2-1}}\le 1+\frac{a+1}{2}\cdot\frac2{\sqrt{3}}<(1+\sqrt{4/5})B
$$
For $B = 2a+1$, we have
$$
1+\frac{B/2+1}{2}\le f(B) = 1+\frac{a^2+3a+2}{\sqrt{4a^2+12a+5}} = 1+\frac{(B/2+1)^2-1/4}{2\sqrt{(B/2+1)^2 - 1}} \le 1+\frac{B/2+1}{2}\cdot \frac{8}{3\sqrt{5}}
$$
Now it is easy to check that $f(B)$ is between $B/4$ and $(1+\sqrt{4/5})B$.
\end{proof}

The proof of Theorem~\ref{th4} also suggests a number $\theta\in\RR$ and a
sequence of positive integers $N_i$ such that the largest gaps in $G(\theta,
N_i)$ tend quickly to $f(B)/N_i$.


\begin{corollary}
Let $B \geq 1$ be an integer and define $\theta = [0, B, 1, B, 1, B, 1,
\ldots]$. Let $p_n/q_n$ be the $n$'th convergent to $\theta$. Then the
largest gap $g$ corresponding to $N =   q_{2n-1} + \lfloor \frac{B+2}{2} \rfloor q_{2n} - 2$ equals $||q_{2n-1}\theta|| - \lfloor \frac{B-2}{2}\rfloor || q_{2n}\theta||$. Further,
as $n \rightarrow \infty$ the quantity $N\cdot H(\theta,N)$ tends to $f(B)$.
\label{maxgap}
\end{corollary}

\section{Kronecker's theorem for badly approximable numbers}
\label{kron-sec}

Equipped with Theorem~\ref{th4} we can improve Theorem~17
from~\cite{Shallit:1996} (see Theorem~\ref{th1} above).
\begin{theorem}
Let $\theta$ be an irrational real number with partial quotients bounded by
$B$, and let $\beta$ be an arbitrary real number.   Suppose $0 \leq \theta,
\beta < 1$. Then there are integers $n, p$ with $n, |p| \leq N  $ such that
$|n \theta - p - \beta| < {{f(B)} \over 2N}$. \label{newkron}
\end{theorem}

\begin{proof}
Apply the three-gap theorem to $\theta$ and $N$.  This involves
sorting the
$N+2$ points $0, \{ \theta \}, \{ 2\theta \}, \ldots, \{ N \theta \}, 1$
in ascending order and creating the $N+1$ intervals between the points,
the union of which forms $[0,1)$.

Now consider the interval in which $\beta$ lies, denote it by $[e_1,e_2]$
with $e_1 = \{n_1\theta\}$ and either $e_2=\{n_2\theta\}$ or $e_2=1$. The
distance from $\beta$ to the closest of these two endpoints is at most
$\frac{e_2-e_1}{2}$, which by Theorem~\ref{badthm} is bounded above by
$f(B)/2N$.

If the closest endpoint to $\beta$ is of the form $\{n\theta\}$ then we
notice that $\{n\theta\} = n\theta - p$ for some integer $p$ with $0 \leq p <
n$. Therefore $|n\theta - p -\beta| \le \frac{f(B)}{2N}$ as required. If 1 is
the closest endpoint to $\beta$ then we take $n=0, p=-1$ and verify $|0\cdot
\theta - 1 - \beta| \le \frac{f(B)}{2N}$.
\end{proof}

\begin{remark}  This improves the bound $(B+2)N^2$ given in
\cite{Shallit:1996}.
\end{remark}

\begin{remark}  The bound in the theorem is tight. We can choose $\theta$ and $N$
from Corollary~\ref{maxgap} and  choose $\beta$ arbitrarily close to the
midpoint of the maximal gap.  By choosing $N$ large enough, we can ensure
that the corresponding maximal gap will be as close as we like to $f(B)/N$. 
\end{remark}

\section{Diversity}
\label{diversity-sec}

We now turn to the application of these results that concerned us
in \cite{Shallit:1996}.

Let ${\bf s} = (s_n)_{n \geq 0}$ be a sequence.  We say that
$\bf s$ is {\it diverse\/}
if, for all $k \geq 2$, the $k$ subsequences
$\{ (s_{ki+a})_{i \geq 0} \ : \ 0 \leq a < k \}$
are all distinct.

If two sequences ${\bf t} = (t_i)_{i \geq 0}$ and
${\bf u} = (u_i)_{i \geq 0}$ are distinct, we define their {\it agreement\/}
$\ag({\bf t},{\bf u})$ to be $\min \{ i \ : \ t_i \not= u_i \}$.
If a sequence $(s_n)_{n \geq 0}$ has the property that
for all $r, a, b$ with $0 \leq a < b < r$ we have
$\ag( (s_{ri+a})_{i \geq 0}, (s_{ri+b})_{i \geq 0} ) \leq f(r)$,
then we say that the function $f$ is a {\it diversity measure\/} for
the sequence $\bf s$.  In \cite{Shallit:1996} it is shown that
there exists a function $f \in O(\log r)$ that
is a diversity measure for almost all binary sequences.  However,
no explicit example of a sequence with this diversity measure is known.

Thus,  it is of interest to produce an explicitly-defined sequence
with slowly-growing diversity measure.  In \cite{Shallit:1996} the
second author looked at the case of Sturmian characteristic sequences
(see, e.g., \cite{Berstel&Seebold:2002}); these are words of the form
${\bf s} = (s_i)_{i \geq 0}$ with
\begin{equation}
s_i = \lfloor (i+2) \theta \rfloor - \lfloor (i+1) \theta \rfloor .
\label{sturm}
\end{equation}
(The indexing here is slightly atypical because we want to index $\bf s$ starting at $0$ instead of the more conventional $1$.)

Suppose $\theta$ has partial quotients bounded by $B$.
Then the second
author proved \cite[Lemma 18]{Shallit:1996} that $4(B+2)^3 r^3$ is a
diversity measure for $(s_i)_{i \geq 0}$.  In this section we
improve this result.

\begin{theorem}
Suppose $0 < \theta < 1$ has partial quotients bounded by $B$.  Consider the
associated Sturmian characteristic sequence $\bf s$ as defined above. Then the
function $2 (B+2)^2 r^2$ is a diversity measure for $\bf s$. \label{newdiver}
\end{theorem}

\begin{proof}
We follow the proof of \cite[Lemma 18]{Shallit:1996} with some small
changes.

We use the ``circular representation'' for subsets of $[0,1)$, identifying
the endpoints $0$ and $1$ and considering each point modulo $1$.

Then $s_{i-1} = 1$ iff $\{ i \theta \} \in [1-\theta, 1)$.  If we can find
$m$ such that $\{ (rm+c) \theta \} \in [1-\theta, 1)$ while $\{ (rm+d) \theta
\} \in [0, 1-\theta)$, then for this $m$ we have $s_{rm+c-1} \not=
s_{rm+d-1}$.

But, using the circular representation of intervals,
$$\{ (rm+c) \theta \} \in [1-\theta, 1) \text{ iff } \{ rm\theta \} \in
I_c := [-(c+1)\theta, -c \theta)$$
and
$$\{ (rm+d) \theta \} \in [0,1-\theta) \text{ iff  } \{ rm\theta \} \in
I_d := [-d\theta, -(d+1)\theta).$$
Furthermore, $\mu(I_c) + \mu(I_d) = 1$ (where $\mu$ is Lebesgue
measure) and so these two intervals have nontrival intersection
if $c \not= d$.  The endpoints of these intervals are of the
form $\{ -i \theta \}$ for some $i$ with $0 \leq i \leq r$.
Then $\mu(I_c \ \cap \ I_d)$ is at least as big as the smallest
gap $g$ in the three-gap theorem corresponding to $N = r$, which
by \cite[Lemma 16]{Shallit:1996}, is at least $1 \over {(B+2)r}$.

Let $m'$ be the midpoint of the interval $I_c \ \cap \ I_d$.
If we could find integers $m, t$ with
$$ |rm \theta  - m' - t | < {1 \over {2 (B+2) r } },$$
then
$$ {1\over 2} \mu (I_c \ \cap \ I_d) \geq
{1 \over 2} g \geq {1 \over {2 (B+2) r }}  > |rm \theta -m' - t|,$$
then $\{ rm\theta \}$ would lie inside $I_c \ \cap \ I_d$.

Since $\theta$ has partial quotients bounded by $B$, we know that $r \theta$
has partial quotients bounded by $r(B+2)$; see, for example
\cite{Lagarias&Shallit:1997}.   Hence by Theorem~\ref{newkron} applied to
$r\theta$, we see that such an $m$ exists with $m \leq r (B+2) f(r(B+2)) \leq
2 (B+2)^2 r^2$.
\end{proof}

We now show that the bound in Theorem~\ref{newdiver} is tight, up to a
constant factor.

Let, as usual, the Fibonacci numbers $F_n$ be defined by
$F_0 = 0$, $F_1 = 1$, and $F_n = F_{n-1} + F_{n-2}$ for $n \geq 2$.
Let the Lucas numbers $L_n$ be defined by
$L_0 = 2$, $L_1 = 1$, and $L_n = L_{n-1} + L_{n-2}$ for $n \geq 2$.
Define $\gamma(n) = \{ n \theta \}$.

In what follows
we let 
\begin{align*}
\alpha &= (1+\sqrt{5})/2 \doteq 1.61803 \\
\beta &= (1-\sqrt{5})/2 \doteq -0.61803 
\end{align*}
be the two zeros of $X^2 - X - 1$, and we
let $\theta = 1/\alpha = -\beta = (\sqrt{5} -1)/2$.
Recall the Binet formulas for the Fibonacci and Lucas numbers:
$F_n = (\alpha^n - \beta^n)/\sqrt{5}$ and $L_n = \alpha^n + \beta^n$.

We start with two useful lemmas.  
\begin{lemma}
\leavevmode
\begin{enumerate}[(i)]
\item $ F_n \theta = F_{n-1} - \beta^n$ for $n \geq 1$;
\item $L_n \theta = L_{n-1} + \sqrt{5} \beta^n$ for $n \geq 1$.
\end{enumerate}
\label{fibfacts}
\end{lemma}

\begin{proof}
Routine manipulation involving the Binet formulas.
\end{proof}

\begin{lemma}
Let $n \geq 2$.
Then 
\begin{align}
\gamma(F_{4n} i + L_{2n} j + F_{2n-1}  )
&= (\gamma(F_{4n})-1)i +  \gamma(L_{2n}) j + \gamma(F_{2n-1}), \quad \text{and}  \label{f1}\\
\gamma( F_{4n} i + L_{2n} j + L_{2n} ) 
&= (\gamma(F_{4n}) - 1)i +  \gamma(L_{2n}) j + \gamma(L_{2n} ) . \label{f2}
\end{align}
for $0 \leq i \leq L_{2n+1} - 2$ and
$0 \leq j \leq F_{2n} - 1$.
\end{lemma}

\begin{proof}
By Lemma~\ref{fibfacts} we get
\begin{align*}
\gamma(F_{2n-1}) &= \theta^{2n-1} \\
\gamma(F_{4n}) &= 1- \theta^{4n} \\
\gamma(L_{2n}) &= \sqrt{5} \theta^{2n},
\end{align*}
which we use repeatedly in what follows.
Then the desired relations \eqref{f1} and \eqref{f2} hold provided there
is no ``wrap-around'' modulo $1$ in those sums.   

We now define two
rectangular arrays $A$ and $B$, as follows:
\begin{align*}
A[i,j] &= (\gamma(F_{4n})-1)i +  \gamma(L_{2n}) j + \gamma(F_{2n-1})  \quad
\text{for
$0 \leq i \leq L_{2n+1}-2$ and $0 \leq j \leq F_{2n} - 1$} \\
B[i,j] &= (\gamma(F_{4n}) - 1)i +  \gamma(L_{2n}) j + \gamma(L_{2n}) 
\quad \text{for 
$0 \leq i \leq L_{2n+1}-2$ and $0 \leq j \leq F_{2n} - 1$} .
\end{align*}
Note that $B[i,j] - A[i,j] = \gamma(L_{2n}) - \gamma(F_{2n-1}) = \theta^{2n+1}$,
which is independent of $i$ and $j$.

We now claim that the entries of $A$ (resp., $B$) can be read in ascending
order as follows:  start at the bottom left (i.e., at
the entry $A[L_{2n+1}-2,0]$, resp., $B[L_{2n+1}-2,0]$).
Proceed up each column to the
top entry.   When you reach the top of a column, continue at the bottom of
the column to its right.  This order is illustrated for $n = 2$ and $A$
in Figure~\ref{fig1}.  As shown in Figure~\ref{fig1}, let $d$ denote the
difference between two entries in the same row, but adjacent 
columns, and let $d'$ denote the difference between two entries
in the same column, but adjacent rows.
Finally, let $d''$ denote the difference between the entry at the
top of a column, and the entry at the bottom of the column to its
right.  
Then it is easy
to verify from the definitions of $A$ and $B$ that
\begin{align*}
d &= \gamma(L_{2n}) = \sqrt{5} \theta^{2n} ; \\
d' &= 1-\gamma(F_{4n}) = \theta^{4n} ; \\
d'' & =  (L_{2n+1} - 2) \gamma(F_{4n}) + \gamma(L_{2n}) =
        \theta^{2n+1} + 2 \theta^{4n} + \theta^{6n+1} .
\end{align*}
\begin{figure}[H]
\begin{center}
\includegraphics[width=6in]{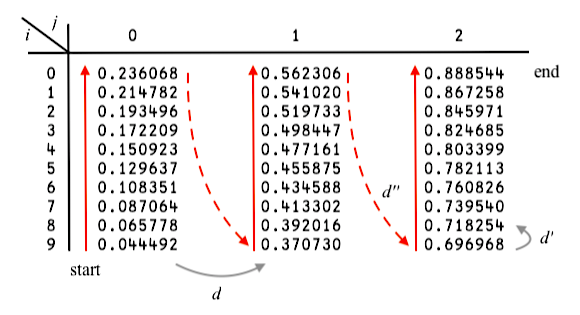}
\end{center}
\caption{The array $A[i,j]$ for $n = 2$}
\label{fig1}
\end{figure}

Finally, we observe that for both 
arrays, every entry is strictly between $0$ and $1$.
By the increasing property of columns and rows,
it suffices to verify this for the entries labeled ``start''
and ``end''.
To see this, note that the entry labeled ``start'' in
$A$ is 
$$A[L_{2n+1} - 2, 0] =
(\gamma(F_{4n}) - 1) (L_{2n+1} - 2) + \gamma(F_{2n-1}) = 2 \theta^{4n} + \theta^{6n+1} > 0,$$
while the entry labeled ``end'' in
$B$ is 
$$B[0, F_{2n} -1] = \gamma(L_{2n}) (F_{2n} - 1) + \gamma(L_{2n}) 
= 1 - \theta^{4n} < 1.$$
Since $A[i,j]<B[i,j]$, the other two entries $A[0, F_{2n}-1]$ and $B[L_{2n+1}-2,0]$ are also between 0 and 1.

Putting all these facts together, we see there is no
``wrap-around'' modulo $1$ in \eqref{f1} and \eqref{f2}, so that
\begin{align*}
A[i,j] &= \gamma(F_{4n} i + L_{2n} j + F_{2n-1}) \\
B[i,j] &= \gamma(F_{4n} i + L_{2n} j + L_{2n})
\end{align*}
for $0 \leq i \leq L_{2n+1}-2$ and
$0 \leq j \leq F_{2n} - 1$.
\end{proof}

We can now prove our lower bound on diversity.
\begin{theorem}
Let $n\geq 2$ be an integer, and let $(s_i)_{i \geq 0}$ be the
Sturmian sequence, defined in \eqref{sturm}, and corresponding to $\theta = {1 \over 2}(\sqrt{5}-1)$.
Let $r = L_{2n}$, $a = F_{2n-1}-1$,
and $b = L_{2n}-1$ for $n \geq 2$.
Then
\begin{align*}
s_{rk + a } &= s_{rk+ b} \text{
for $0 \leq k \leq F_{4n+1} - F_{2n+1}-2$, but } \\
0 = s_{rk+a } &\not= s_{rk+b} = 1 \text{
for $k = F_{4n+1} - F_{2n+1}- 1$} .
\end{align*}
Since $F_{4n+1}/L_{2n}^2 \approx (\sqrt{5} + 10)/10$, this implies
that $\bf s$ has diversity measure $\Omega(r^2)$.
\end{theorem}

\begin{proof}
From the definition of $s_n$, we have 
that $s_n = 0$ iff $\gamma(n+1) \in (1-\theta, 1)$.
Thus, for the choices of $r, a, b$ above, it suffices to find
the smallest $k\geq 0$ such that $\gamma(rk+a+1) \not\in (1-\theta,1)$, but
$\gamma(rk+b+1) \in (1-\theta,1)$ (or vice versa).

Let's look again at the arrays $A$ and $B$.   
We claim that $A$ contains each value
$\gamma(rk+a+1)$ for $0 \leq k \leq F_{4n+1} - F_{2n} - 2$ exactly once
and similarly $B$ contains each value
$\gamma(rk+b+1)$ for $0 \leq k \leq F_{4n+1} - F_{2n} - 2$ 
exactly once.
To see this, it suffices to show that
the numbers
$$ F_{4n} i + L_{2n} j $$
for $0 \leq i \leq L_{2n+1} -2$, $0 \leq j \leq F_{2n} - 1$,
represent each multiple $k \cdot L_{2n}$,
for $0 \leq k \leq F_{4n+1} - F_{2n} - 2$, exactly once.
This immediately follows from the observation that
$F_{4n} = F_{2n} L_{2n}$, so (in effect) we are representing
$k$ in the mixed radix system with
place values $(1, L_{2n}, F_{2n})$.

Now we have already observed above that
$B[i,j] - A[i,j] > 0$, so if there exists $k$ with
$0 \leq k \leq F_{4n+1} - F_{2n} - 2$ such that
$s_{rk+a} \not= s_{rk+b}$, then it must be that
$\gamma(rk+a+1) \in (0,1-\theta)$ and
$\gamma(rk+b+1) \in (1-\theta, 1)$.
So we need to find $i,j$ such that
$A[i,j] < 1 - \theta < B[i,j]$.
We claim this occurs uniquely for
$i = L_{2n+1} - 2$ and $j = F_{2n-2}$.  

We can check now that, for these values of $i$ and $j$, we have
\begin{align*}
A[i,j] &= \theta^2 + 2 \theta^{4n} + \theta^{6n+1} - \theta^{4n-2} \\
B[i,j] &= \theta^2 + \sqrt{5} \theta^{2n} + 2 \theta^{4n} +
\theta^{6n+1} - \theta^{2n-1} - \theta^{4n-2} ,
\end{align*}
and so we see
\begin{align}
\theta^2 - \theta^{4n-2} &< A[i,j] < \theta^2 ; \label{e1} \\
\theta^2 < B[i,j] &< \theta^2 + \theta^{2n-3} + 3\theta^{4n} , \label{e2}
\end{align}
giving us the desired inequalities, because $1-\theta = \theta^2$.
It now follows that $s_{i F_{4n} + j L_{2n} + F_{2n-1}-1} = 0$,
while $s_{i F_{4n} + j L_{2n} + L_{2n}-1} = 1$.

It remains to see this is the only possible choice $(i,j)$ for
$A[i,j]<1-\theta< B[i,j]$.  This follows from
the estimates \eqref{e1} and \eqref{e2}, combined with
our calculations of $d, d', d''$ given above:   moving to an adjacent
row or column, or to the top of the previous column, violates the
desired inequality.

Finally, our claim follows from the identity
$$ (L_{2n+1} - 2) F_{4n} + F_{2n-2} L_{2n} = L_{2n} (F_{4n+1} - F_{2n+1} - 1),$$
which can be easily checked.
\end{proof}

\end{document}